\documentclass[10pt]{amsart}
\usepackage{amssymb,amstext,amsmath,amscd,amsthm,amsfonts,enumerate,latexsym,stmaryrd,multicol,geometry,graphicx}
\usepackage[usenames]{color}
\usepackage[all]{xy}
\newtheorem{thm}{Theorem}[section]
\newtheorem{lem}[thm]{Lemma}
\newtheorem{prop}[thm]{Proposition}
\newtheorem{cor}[thm]{Corollary}
\theoremstyle{definition}
\newtheorem{dfn}[thm]{Definition}
\newtheorem{ques}[thm]{Question}

\newtheorem{rem}[thm]{Remark}

\theoremstyle{remark}

\newtheorem*{claim*}{Claim}
\newtheorem*{ac}{Acknowlegments}

\numberwithin{equation}{thm}
\def\ab{(\mathbf{ab})}
\def\add{\operatorname{add}}
\def\ass{\operatorname{Ass}}
\def\C{\mathcal{C}}
\def\ch{\operatorname{char}}
\def\ci{(\mathbf{ci})}
\def\cm{(\mathbf{cm})}

\def\cone{\operatorname{cone}}
\def\dep{(\mathbf{dep})}
\def\depth{\operatorname{depth}}
\def\E{\mathbf{E}}
\def\EE{\operatorname{E}}
\def\ee{(\mathbf{ee})}
\def\Ext{\operatorname{Ext}}
\def\gap{(\mathbf{gap})}
\def\gdim{\operatorname{Gdim}}
\def\ge{\geqslant}
\def\gor{(\mathbf{gor})}
\def\Hom{\operatorname{Hom}}
\def\id{\operatorname{id}}
\def\K{\operatorname{\mathcal{K}}}
\def\k{\operatorname{K}}
\def\ka{\operatorname{K_{ac}}}
\def\kt{\operatorname{K_{tac}}}
\def\le{\leqslant}
\def\lten{\otimes^\mathbf{L}}
\def\M{\mathfrak{M}}
\def\m{\mathfrak{m}}
\def\mod{\operatorname{mod}}
\def\nf{\operatorname{NF}}
\def\P{\mathbf{P}}
\def\p{\mathfrak{p}}
\def\Proj{\operatorname{Proj}}
\def\proj{\operatorname{proj}}
\def\Q{\mathbf{Q}}
\def\res{\operatorname{res}}

\def\spec{\operatorname{Spec}}
\def\syz{\Omega}
\def\tac{(\mathbf{tac})}
\def\te{(\mathbf{te})}
\def\Tor{\operatorname{Tor}}
\def\TR{(\mathbf{tr})}
\def\tr{\operatorname{Tr}}
\def\X{\mathcal{X}}
\def\xx{\boldsymbol{x}}
\begin{document}
\title[On the vanishing of Ext modules over a local UFD with an isolated singularity]{On the vanishing of Ext modules over a local unique factorization domain with an isolated singularity}
\author{Kaito Kimura}
\address[Kimura]{Graduate School of Mathematics, Nagoya University, Furocho, Chikusaku, Nagoya 464-8602, Japan}
\email{m21018b@math.nagoya-u.ac.jp}
\author{Justin Lyle}
\address[Lyle]{Department of Mathematical Sciences, University of Arkansas, 525 Old Main, Fayetteville, AR 72701, USA}
\email{jlyle106@gmail.com}
\urladdr{https://jlyle42.github.io/justinlyle/}
\author{Yuya Otake}
\address[Otake]{Graduate School of Mathematics, Nagoya University, Furocho, Chikusaku, Nagoya 464-8602, Japan}
\email{m21012v@math.nagoya-u.ac.jp}
\author{Ryo Takahashi}
\address[Takahashi]{Graduate School of Mathematics, Nagoya University, Furocho, Chikusaku, Nagoya 464-8602, Japan}
\email{takahashi@math.nagoya-u.ac.jp}
\urladdr{https://www.math.nagoya-u.ac.jp/~takahashi/}
\thanks{Kimura was partly supported by Grant-in-Aid for JSPS Fellows 23KJ1117. Otake was partly supported by Grant-in-Aid for JSPS Fellows 23KJ1119. Takahashi was partly supported by JSPS Grant-in-Aid for Scientific Research 23K03070}
\begin{abstract}
This paper provides a method to get a noetherian equicharacteristic local UFD with an isolated singularity from a given noetherian complete equicharacteristic local ring, preserving certain properties.
This is applied to invesitgate the (non)vanishing of Ext modules.
It is proved that there exist a Gorenstein local UFD $A$ having an isolated singularity such that $\Ext_A^{\gg0}(M,N)=0$ does not imply $\Ext_A^{\gg0}(N,M)=0$, a Gorenstein local UFD $B$ having an isolated singularity such that $\Tor_{>0}^B(M,N)=0$ does not imply $\depth(M\otimes_BN)=\depth M+\depth N-\depth B$, and a Cohen--Macaulay local UFD $C$ having an isolated singularity such that $\Ext_C^{>0}(M,C)=0$ does not imply the total reflexivity of $M$.
\end{abstract}
\maketitle
\section{Introduction}

Throughout the present paper, we assume that all rings are commutative noetherian rings with identity.
For a local ring $R$ we denote by $\m_R$ the unique maximal ideal of $R$.

Let $\P$ and $\Q$ be properties of local rings such that the following implications hold for each local ring $R$:
\begin{equation}\label{6}
\begin{array}{rlrll}
\text{$R[\![X]\!]$ satisfies $\P$}
&\Longrightarrow&\text{$R$ satisfies $\P$}
&\Longrightarrow&\text{$\widehat R$ satisfies $\P$}\,,\\
\text{$R[\![X]\!]$ satisfies $\Q$}
&\Longleftarrow&\text{$R$ satisfies $\Q$}
&\Longleftarrow&\text{$\widehat R$ satisfies $\Q$}\,.
\end{array}
\end{equation}
Here, $R[\![X]\!]$ denotes a formal power series ring over $R$, while $\widehat R$ stands for the $\m_R$-adic completion of $R$.
The main result of this paper is the following theorem concerning rings that do not satisfy $\P$ but do satisfy $\Q$.

\begin{thm}\label{0}
Suppose that there exists an equicharacteristic complete local ring $(A,\m_A,k)$ of depth $\alpha$ which does not satisfy $\P$ but satisfies $\Q$.
Let $\rho\ge\max\{2,\alpha\}$ and $\sigma\ge\max\{1,\alpha\}$ be integers.
Then there exist:
\begin{enumerate}[\rm(1)]
\item
an equicharacteristic local unique factorization domain $(R,\m_R,k)$ of depth $\rho$ with an isolated singularity which does not satisfy $\P$ but satisfies $\Q$, and
\item
an equicharacteristic local domain $(S,\m_S,k)$ of depth $\sigma$ with an isolated singularity which does not satisfy $\P$ but satisfies $\Q$.
\end{enumerate}
\end{thm}

In this paper, we are interested in applying the above theorem to investigate the vanishing of Ext modules over a local (unique factorization) domain with an isolated singularity.

A ring $R$ is said to satisfy $\ee$ if for all finitely generated $R$-modules $M$ and $N$ such that $\Ext_R^{\gg0}(M,N)=0$ one has $\Ext_R^{\gg0}(N,M)=0$.
Huneke and Jorgensen \cite{HJ} introduce the notion of an {\em AB ring}, and prove that any AB ring satisfies $\ee$.
The relationships of $\ee$ with several other properties are described in \cite[(4.15)]{JS5}.

We say that $R$ satisfies $\dep$ if all finitely generated $R$-modules $M$ and $N$ with $\Tor_{>0}^R(M,N)=0$ satisfy the {\em depth formula}, i.e.,
$$
\depth(M\otimes_RN)=\depth M+\depth N-\depth R.
$$
Christensen and Jorgensen \cite{CJ} prove that every AB ring satisfies $\dep$.
On the other hand, it has been an open question for several decades now whether every local ring, or even every Gorenstein local ring, satisfies $\dep$, and much work has been put towards providing sufficient conditions for $\dep$ to hold; see \cite{AY,CL,CJ,ST} for but a few examples, and see the introduction of \cite{BJ} for an overview on the history of this problem. In this work, we provide a negative answer to this question. Our result in this direction comes as a consequence of the following theorem.

\begin{thm}\label{27}
A Gorenstein local ring of positive dimension satisfies $\dep$ if and only if it is an AB ring.
\end{thm}

The property $\ee$ implies Gorensteinness.
A result that negates the converse is given by Jorgensen and \c{S}ega \cite{JS5}, where they construct, for a field $k$ that is not algebraic over a finite field, an artinian Gorenstein equicharacteristic local ring $(A,\m_A,k)$ which does not satisfy the property $\ee$.
(Note that this also shows that the ``only if'' part of Theorem \ref{27} does not necessarily hold true without the assumption that the ring has positive dimension.)
One can apply Theorem \ref{0} to this local ring $A$ and the properties $\P=\ee$ and $\Q=\text{Gorensteinness}$ and use Theorem \ref{27} to obtain the following corollary. 

\begin{cor}\label{4}
Let $k$ be a field which is not algebraic over a finite field.
\begin{enumerate}[\rm(1)]
\item
For every $d\ge2$, there is a $d$-dimensional Gorenstein equicharacteristic local unique factorization domain $(R,\m_R,k)$ with an isolated singularity which does not satisfy $\ee$.
Hence, $R$ does not satisfy $\dep$.
\item
There exists a $1$-dimensional Gorenstein equicharacteristic local domain $(S,\m_S,k)$ that does not satisfy $\ee$.
Hence, $S$ does not satisfy $\dep$.
\end{enumerate}
In particular, $R$ is a non-AB unique factorization domain, and $S$ is a non-AB domain of dimension one.
\end{cor}

We say that $R$ satisfies $\TR$ if a finitely generated $R$-module $M$ is totally reflexive whenever $\Ext_R^{>0}(M,R)=0$; recall that $M$ is called {\em totally reflexive} if the natural map $M\to M^{\ast\ast}$ is an isomorphism and $\Ext_R^{>0}(M,R)=\Ext_R^{>0}(M^\ast,R)=0$, where $(-)^\ast$ is the $R$-dual functor.
The property $\TR$ is the same as the {\em weakly Gorenstein} property in the sense of Ringel and Zhang \cite{RZ}.
A (chain) complex of projective $R$-modules is called {\em totally acyclic} if it and its $R$-dual are both acyclic.
We say that $R$ satisfies $\tac$ if every acyclic complex of finitely generated projective $R$-modules is totally acyclic.
This is a finitely generated module version of the property studied by Iyengar and Krause \cite{IK}.
The following describes the relationships between $\TR$, $\tac$ and $\dep$.

\begin{thm}\label{29}
A local ring satisfying $\tac$ satisfies $\TR$.
A Cohen--Macaulay local ring of positive dimension satisfying $\dep$ satisfies $\TR$.
\end{thm}

Since every Gorenstein ring satisfies $\TR$, the rings $R$ and $S$ given in Corollary \ref{4} show that the converse of the second assertion of Theorem \ref{29} does not hold in general.

Jorgensen and \c{S}ega \cite{JS6} construct, for a field $k$ that is not algebraic over a finite field, an artinian equicharacteristic local ring $(A,\m_A,k)$ which does not satisfy $\TR$.
(This shows that the second assertion of Theorem \ref{29} does not necessarily hold without the assumption of positive dimension.)
Applying Theorem \ref{0} to this local ring $A$ and the properties  $\P=\TR$ and $\Q=\text{Cohen--Macaulayness}$ and using Theorem \ref{29}, we obtain the corollary below.
It is claimed in \cite[Theorem 1.1]{Y} that every generically Gorenstein ring satisfies $\tac$, and it is claimed in \cite[Corollary 1.3]{Y} that every generically Gorenstein ring satisfies $\TR$.
Here, a {\em generically Gorenstein} ring is defined as a ring which is locally Gorenstein on the associated prime ideals. 
Since every domain is a generically Gorenstein ring, these claims turn out to be incorrect in any positive dimension.

\begin{cor}\label{5}
Let $k$ be a field which is not algebraic over a finite field.
\begin{enumerate}[\rm(1)]
\item
For every integer $d\ge2$, there exists a $d$-dimensional Cohen--Macaulay equicharacteristic local unique factorization domain $(R,\m_R,k)$ with an isolated singularity which does not satisfy $\TR$.
Therefore, $R$ does not satisfy $\tac$.
Moreover, $R$ is a non-Gorenstein ring that does not satisfy $\dep$.
\item
There is a $1$-dimensional Cohen--Macaulay equicharacteristic local domain $(S,\m_S,k)$ which does not satisfy $\TR$.
Hence, $S$ does not satisfy $\tac$.
Also, $S$ is a non-Gorenstein ring that does not satisfy $\dep$.
\end{enumerate}
In particular, the rings $R$ and $S$ give counterexamples to \cite[Theorem 1.1 and Corollary 1.3]{Y}.
\end{cor}

The paper is organized as follows.
In Section 2, we show that several properties of local rings including the above mentioned ones are preserved under fundamental operations.
In Section 3, we shall investigate the relationships between $\TR$, $\dep$, the AB property and a question of Araya.
In Section 4, we provide all of the proofs of those three results stated above, together with a couple of observations on related topics.

\begin{ac}
The authors thank Yuji Yoshino for reading this paper and giving them valuable comments.
\end{ac}

\section{Properties preserved under basic operations}

First of all, for ease of reference, we make a list of those properties which we handle in this paper.

\begin{dfn}
We consider the following conditions for a (commutative noetherian) local ring $R$.
\begin{enumerate}
\item[$\ab$]
The local ring $R$ is an {\em AB ring}, that is, $R$ is Gorenstein and there exists an integer $n$ such that whenever one has $\Ext_R^{\gg0}(M,N)=0$ for finitely generated $R$-modules $M,N$ it holds that $\Ext_R^{>n}(M,N)=0$.
\item[$\ci$]
The local ring $R$ is a complete intersection.
\item[$\cm$]
The local ring $R$ is Cohen--Macaulay.
\item[$\dep$]
All nonzero finitely generated $R$-modules $M$ and $N$ such that $\Tor_{>0}^R(M,N)=0$ satisfy the depth formula $\depth(M\otimes_RN)=\depth M+\depth N-\depth R$.
\item[$\ee$]
For all finitely generated $R$-modules $M$ and $N$ such that $\Ext_R^{\gg0}(M,N)=0$ one has $\Ext_R^{\gg0}(N,M)=0$.
\item[$\gap$]
The local ring $R$ is Gorenstein and has finite Ext-gap, where the {\em Ext-gap} of $R$ is defined to be the supremum of integers $g\ge0$ such that there exist finitely generated $R$-modules $M,N$ and an integer $n\ge0$ such that $\Ext_R^i(M,N)$ does not vanish for $i=n,n+g+1$ but does for $i=n+1,\dots,n+g$.
\item[$\gor$]
The local ring $R$ is Gorenstein.
\item[$\tac$]
Every acyclic complex of finitely generated projective $R$-modules is totally acyclic.
\item[$\te$]
For any two finitely generated $R$-modules $M,N$ with $\Tor_{\gg0}^R(M,N)=0$ one has $\Ext_R^{\gg0}(M,N)=0$.
\item[$\TR$]
A finitely generated $R$-module $M$ is totally reflexive whenever $\Ext_R^{>0}(M,R)=0$.
\end{enumerate}
\end{dfn}

Throughout this section, let $R$ be a local ring with maximal ideal $\m$ and residue field $k$.
We first deal with the property $\ee$.
We establish two lemmas to prove a proposition on $\ee$.

\begin{lem}\label{2}
Let $M$ be an $\widehat R$-module of finite length.
Then $M$ has finite length as an $R$-module as well.
The maps $f:M\rightleftarrows M\otimes_R\widehat R:g$, given by $x\mapsto x\otimes1$ and $ax\mapsfrom x\otimes a$, are mutually inverse $\widehat R$-isomorphisms.
\end{lem}

\begin{proof}
By assumption, there is a composition series $0=M_0\subsetneq\cdots\subsetneq M_n=M$ of the $\widehat R$-module $M$, where $n=\ell_{\widehat R}(M)$.
Then for each $i$ there is an $\widehat R$-isomorphism $M_i/M_{i-1}\cong\widehat R/\widehat\m\cong k$.
Hence this is also a composition series of the $R$-module $M$, and we have $\ell_R(M)=\ell_{\widehat R}(M)=n$.
Thus, $M$ has finite length as an $R$-module.
In particular, $M$ is finitely generated and complete as an $R$-module, so that the map $f$ is an $R$-isomorphism.
It is easy to see that $g$ is an $\widehat R$-homomorphism and $gf=\id_M$.
Therefore, $f$ and $g$ are mutually inverse $\widehat R$-isomorphisms.
\end{proof}

Let $M$ be a finitely generated $R$-module.
Take a minimal free resolution $\cdots\xrightarrow{\partial_3}F_2\xrightarrow{\partial_2}F_1\xrightarrow{\partial_1}F_0\to M\to0$ of $M$.
The image of the $i$th differential map $\partial_i$ is called the {\em $i$th syzygy} of $M$ and denoted by $\syz^iM$ (or $\syz_R^iM$ to specify the base ring $R$).
We set $\syz^0M=M$ and $\syz M=\syz^1M$.
Note that $\syz^iM$ is uniquely determined by $M$ and $i$ up to isomorphism.

\begin{lem}\label{1}
Let $M$ and $N$ be finitely generated modules over the local ring $R$.
\begin{enumerate}[\rm(1)]
\item
Assume $R$ is Gorenstein.
Let $m,n\ge0$.
Then $\Ext_R^{\gg0}(M,N)=0$ if and only if $\Ext_R^{\gg0}(\syz^mM,\syz^nN)=0$.
\item
Let $\xx=x_1,\dots,x_n$ be a regular sequence on $M$ (resp. $N$).
One then has that $\Ext_R^{\gg0}(M,N)=0$ if and only if $\Ext_R^{\gg0}(M/\xx M,N)=0$ (resp. $\Ext_R^{\gg0}(M,N/\xx N)=0$).
\end{enumerate}
\end{lem}

\begin{proof}
(1) Fix an integer $i>0$.
We have $\Ext_R^i(\syz^mM,N)\cong\Ext_R^{i+m}(M,N)$.
There is an exact sequence $0\to\syz^nN\to F_{n-1}\to\cdots\to F_0\to N\to0$ of finitely generated $R$-modules with $F_0,\dots,F_{n-1}$ free.
Decomposing this into short exact sequences and using the fact from the Gorensteinness of $R$ that $\Ext_R^{>d}(\syz^mM,R)=0$ with $d=\dim R$, we observe that $\Ext_R^{>t}(\syz^mM,N)=0$ if and only if $\Ext_R^{>t+n}(\syz^mM,\syz^nN)=0$ for each $t\ge d$.
Therefore, $\Ext_R^{\gg0}(M,N)=0$ if and only if $\Ext_R^{\gg0}(\syz^mM,N)=0$, if and only if $\Ext_R^{\gg0}(\syz^mM,\syz^nN)=0$.

(2) We only show $\Ext_R^{\gg0}(M,N)=0$ if and only if $\Ext_R^{\gg0}(M/\xx M,N)=0$ for an $M$-regular sequence $\xx=x_1,\dots,x_n$; the other assertion is shown similarly.
By induction on $n$, it suffices to prove that $\Ext_R^{\gg0}(M,N)=0$ if and only if $\Ext_R^{\gg0}(M/xM,N)=0$ for an $M$-regular element $x\in\m$.
This is actually a consequence of the application of $\Ext_R^{\gg0}(-,N)$ to the exact sequence $0\to M\xrightarrow{x}M\to M/xM\to0$, plus Nakayama's lemma.
\end{proof}

The proposition below collects statements on $\ee$; it is preserved under several standard operations.

\begin{prop}\label{7}
\begin{enumerate}[\rm(1)]
\item
If $R$ satisfies the property $\ee$, then $R$ is a Gorenstein ring.
\item
Let $R\to S$ be a flat local homomorphism of local rings.
If $S$ satisfies $\ee$, so does $R$.
\item
Let $x\in\m$ be an $R$-regular element.
If $R/(x)$ satisfies $\ee$, then so does $R$.
\item
If the local ring $R$ satisfies the property $\ee$, then so does the $\m$-adic completion $\widehat R$ of $R$.
\end{enumerate} 
\end{prop}

\begin{proof}
(1) As $\Ext_R^{>0}(R,k)=0$, the condition $\ee$ implies $\Ext_R^{\gg0}(k,R)=0$, which means that $R$ is Gorenstein.

(2) Let $M,N$ be finitely generated $R$-modules such that $\Ext_R^i(M,N)=0$ for $i\gg0$.
Then $\Ext_S^i(M\otimes_RS,N\otimes_RS)\cong\Ext_R^i(M,N)\otimes_RS=0$ for $i\gg0$.
As $S$ satisfies $\ee$ and $M\otimes_RS,N\otimes_RS$ are finitely generated $S$-modules, we have $\Ext_S^i(N\otimes_RS,M\otimes_RS)=0$ for $i\gg0$.
Hence $\Ext_R^i(N,M)\otimes_RS=0$ for $i\gg0$.
Since $S$ is faithfully flat over $R$, we get $\Ext_R^i(N,M)=0$ for $i\gg0$.

(3) As $R/(x)$ satisfies $\ee$, it is Gorenstein by (1), and so is $R$.
Let $M,N$ be finitely generated $R$-modules with $\Ext_R^{\gg0}(M,N)=0$.
Lemma \ref{1}(1) implies $\Ext_R^{\gg0}(\syz M,\syz N)=0$. Note that $x$ is regular on $\syz M$ and $\syz N$.
Lemma \ref{1}(2) shows $\Ext_R^{\gg0}(\syz M/x\syz M,\syz N)=0$.
As $\Ext_{R/(x)}^i(\syz M/x\syz M,\syz N/x\syz N)$ is isomorphic to $\Ext_R^{i+1}(\syz M/x\syz M,\syz N)$ by \cite[\S18, Lemma 2]{M}, we have $\Ext_{R/(x)}^{\gg0}(\syz M/x\syz M,\syz N/x\syz N)=0$.
Since $R/(x)$ satisfies $\ee$, it follows that $\Ext_{R/(x)}^{\gg0}(\syz N/x\syz N,\syz M/x\syz M)=0$.
Hence $\Ext_R^{\gg0}(\syz N/x\syz N,\syz M)=0$. We get $\Ext_R^{\gg0}(\syz N,\syz M)=0$ by Lemma \ref{1}(2), and $\Ext_R^{\gg0}(N,M)=0$ by Lemma \ref{1}(1).

(4) Since $R$ satisfies $\ee$, it is Gorenstein by (1), and so is $\widehat R$.
Let $M,N$ be finitely generated $\widehat R$-modules with $\Ext_{\widehat R}^{\gg0}(M,N)=0$.
Choose integers $m,n\ge0$ such that $Z=\syz_{\widehat R}^mM$ and $Y=\syz_{\widehat R}^nN$ are maximal Cohen--Macaulay $\widehat R$-modules.
Lemma \ref{1}(1) shows $\Ext_{\widehat R}^{\gg0}(Z,Y)=0$.
Let $\xx=x_1,\dots,x_d$ be a system of parameters of $R$.
Then $\xx$ is a regular sequence on $R$ and $\widehat R$.
By Lemma \ref{1}(2) we have $\Ext_{\widehat R}^{\gg0}(V,W)=0$, where $V=Z/\xx Z$ and $W=Y/\xx Y$.
Since $V,W$ have finite length as $\widehat R$-modules, Lemma \ref{2} implies
\begin{equation}\label{3}
\Ext_R^i(V,W)\otimes_R\widehat R
\cong\Ext_{\widehat R}^i(V\otimes_R\widehat R,W\otimes_R\widehat R)
\cong\Ext_{\widehat R}^i(V,W)=0\text{ for }i\gg0.
\end{equation}
The faithful flatness of $\widehat R$ over $R$ implies $\Ext_R^{\gg0}(V,W)=0$.
As $V,W$ are finitely generated $R$-modules and $R$ satisfies $\ee$, we get $\Ext_R^{\gg0}(W,V)=0$.
Similarly as in \eqref{3} we have $\Ext_{\widehat R}^i(W,V)\cong\Ext_R^i(W,V)\otimes_R\widehat R=0$ for $i\gg0$.
We obtain $\Ext_{\widehat R}^{\gg0}(Y,Z)=0$ by Lemma \ref{1}(2), and $\Ext_{\widehat R}^{\gg0}(N,M)=0$ by Lemma \ref{1}(1).
\end{proof}

A similar argument to the above proof shows that the same statement as Proposition \ref{7} holds for the property $\te$.

\begin{prop}\label{7'}
\begin{enumerate}[\rm(1)]
\item
If $R$ satisfies the property $\te$, then $R$ is a Gorenstein ring.
\item
Let $R\to S$ be a flat local homomorphism of local rings.
If $S$ satisfies $\te$, so does $R$.
\item
Let $x\in\m$ be an $R$-regular element.
If $R/(x)$ satisfies $\te$, then so does $R$.
\item
If the local ring $R$ satisfies the property $\te$, then so does the $\m$-adic completion $\widehat R$ of $R$.
\end{enumerate} 
\end{prop}

Next we deal with the property $\TR$.
For a finitely generated module $M$ over a ring $R$ we denote by $\gdim_RM$ the G-dimension of $M$.
Note that $M$ is totally reflexive if and only if $\gdim_RM\le0$.
For the details of G-dimension, we refer the reader to \cite{AB,C}.
The property $\TR$ is retained under taking a local flat extension and modding out by a regular element.

\begin{prop}\label{8}
\begin{enumerate}[\rm(1)]
\item
Let $R\to S$ be a flat local homomorphism of local rings.
If $S$ satisfies $\TR$, so does $R$.
\item
Let $x\in\m$ be an $R$-regular element.
Then $R$ satisfies $\TR$ if and only if so does $R/(x)$.
\end{enumerate}
\end{prop}

\begin{proof}
(1) Let $M$ be a finitely generated $R$-module such that $\Ext_R^{>0}(M,R)=0$.
Then $M\otimes_RS$ is a finitely generated $S$-module with $\Ext_S^{>0}(M\otimes_RS,S)=0$, as $S$ is flat over $R$.
Since $S$ satisfies $\TR$, the $S$-module $M\otimes_RS$ is totally reflexive.
As $S$ is faithfully flat over $R$, the $R$-module $M$ is totally reflexive by \cite[(4.1.4)]{AF}.

(2) The `only if' part:
Let $M$ be a finitely generated $R/(x)$-module such that $\Ext_{R/(x)}^{>0}(M,R/(x))=0$.
Then $\Ext_R^{>1}(M,R)=0$ by \cite[\S18, Lemma 2]{M}, and hence $\Ext_R^{>0}(\syz_RM,R)=0$.
Since the ring $R$ satisfies $\TR$, the $R$-module $\syz_RM$ is totally reflexive, so that $\gdim_RM\le1$ by \cite[(1.2.9)]{C}.
We get $\gdim_{R/(x)}M=\gdim_RM-1\le0$ by \cite[(1.5.3)]{C}, and therefore $M$ is totally reflexive over $R/(x)$.

The `if' part:
Let $M$ be a finitely generated $R$-module with $\Ext_R^{>0}(M,R)=0$.
Then $\Ext_R^{>0}(\syz_RM,R)=0$.
As $x$ is $\syz_RM$-regular, there is an exact sequence $0\to\syz_RM\xrightarrow{x}\syz_RM\to\syz_RM/x\syz_RM\to0$, which shows that $\Ext_R^{>1}(\syz_RM/x\syz_RM,R)=0$.
Using \cite[\S18, Lemma 2]{M} again, we get $\Ext_{R/(x)}^{>0}(\syz_RM/x\syz_RM,R/(x))=0$.
Since $R/(x)$ satisfies $\TR$, the $R/(x)$-module $\syz_RM/x\syz_RM$ is totally reflexive.
By \cite[(1.2.9)\&(1.4.5)]{C} we get
$$
\gdim_RM-1\le\gdim_R(\syz_RM)=\gdim_{R/(x)}(\syz_RM/x\syz_RM)\le0,
$$
and hence $\gdim_RM\le1$.
It follows by \cite[(1.2.7)]{C} or \cite[(4.1.3)]{AF} that $M$ is totally reflexive as an $R$-module.
\end{proof}

Now we consider the property $\tac$.
We first remark that it can be described in terms of frequently used notations of homotopy categories.

\begin{rem}
Denote by $\Proj R$ the category of projective $R$-modules, and by $\proj R$ the category of finitely generated projective $R$-modules.
For $\C\in\{\Proj R,\proj R\}$, denote by $\k(\C)$ the homotopy category of complexes over $\C$, by $\ka(\C)$ the full subcategory of $\k(\C)$ consisting of acyclic complexes, and by $\kt(\C)$ the full subcategory of $\k(\C)$ consisting of totally acyclic complexes. 
Then by definition one has:
\begin{equation}\label{11}
\text{$R$ satisfies $\tac$}\iff\ka(\proj R)=\kt(\proj R).
\end{equation}
This equivalence should be compared with the equivalence below, which holds when $R$ admits a dualizing complex and is due to Iyengar and Krause \cite[Corollary 5.5]{IK}.
\begin{equation}\label{10}
\text{$R$ is Gorenstein}\iff\ka(\Proj R)=\kt(\Proj R).
\end{equation}
\end{rem}

We say that an $R$-module $M$ is an {\em $\infty$-syzygy} if there exists an exact sequence $0\to M\to F^0\to F^1\to F^2\to\cdots$ of finitely generated free $R$-modules.
We set $(-)^\ast=\Hom_R(-,R)$ and denote by $\tr(-)$ the (Auslander) transpose.
The following proposition tells us the relationship of $\tac$ with other properties including $\TR$.

\begin{prop}\label{9}
The following implications hold true.
$$
\begin{array}{l}
\text{$R$ is Gorenstein}\ \Rightarrow\ 
\text{$R$ satisfies $\tac$}\ \Leftrightarrow\ 
\text{any $\infty$-syzygy is totally reflexive}\\
\phantom{\text{$R$ is Gorenstein}\ \Rightarrow\ 
\text{$R$ satisfies $\tac$}}\ \Leftrightarrow\ 
\text{any $\infty$-syzygy $M$ satisfies $\Ext_R^1(M,R)=0$}\ \Rightarrow\ 
\text{$R$ satisfies $\TR$.}
\end{array}
$$
\end{prop}

\begin{proof}
Call the five conditions (1)--(5) in order.
Combining \eqref{11} and \eqref{10} shows (1) implies (2).

Let $M$ be an $\infty$-syzygy.
There is an exact sequence $0\to M\to F^0\to F^1\to\cdots$ of finitely generated free $R$-modules.
Splicing this with a minimal free resolution $\cdots\to F_1\to F_0\to M\to0$, we get an acyclic complex $(\cdots\to F_1\to F_0\to F^0\to F^1\to\cdots)$ of finitely generated free $R$-modules.
If this is totally acyclic, then $M$ is totally reflexive by \cite[(4.2.6)]{C}.
Thus (2) implies (3).
By the definition of total reflexivity, (3) implies (4).

Let $P=(\cdots\to P^i\xrightarrow{\partial^i}P^{i+1}\to\cdots)$ be an acyclic complex of finitely generated free $R$-modules.
Then for each $i$ the image $C^i$ of the map $\partial^i$ is an $\infty$-syzygy.
If $\Ext_R^1(C^i,R)=0$ for all $i$, then the complex $P^\ast$ is exact, and hence $P$ is totally acyclic.
Hence, (4) implies (2).

Let $N$ be a finitely generated $R$-module such that $\Ext_R^{>0}(N,R)=0$.
Then $\tr N$ is an $\infty$-syzygy by \cite[(2.17)]{AB}.
If $\tr N$ is totally reflexive, then so is $N$.
Therefore, (3) implies (5).
\end{proof}

Using the above proposition, we obtain the following corollary regarding the property $\tac$.

\begin{cor}\label{16}
Let $x\in\m$ be an $R$-regular element.
Then $R$ satisfies $\tac$ if and only if so does $R/(x)$.
\end{cor}

\begin{proof}
Let $M$ be an $\infty$-syzygy over $R$.
Then there exists an exact sequence $0\to M\to F^0\to F^1\to\cdots$ of finitely generated $R$-modules such that each $F^i$ is free.
Note then that $x$ is regular on $M$ and each $F^i$.
Tensoring $R/(x)$ over $R$ gives rise to an exact sequence $0\to M/xM\to F^0/xF^0\to F^1/xF^1\to\cdots$ of finitely generated $R/(x)$-modules and each $F^i/xF^i$ is free over $R/(x)$.
Hence $M/xM$ is an $\infty$-syzygy over $R/(x)$.
If $M/xM$ is totally reflexive over $R/(x)$, then $M$ is totally reflexive over $R$ by \cite[(1.4.4)]{C}.

Let $M$ be an $\infty$-syzygy over $R/(x)$.
Then there exists an exact sequence $0\to M\to P^0\to P^1\to\cdots$ of finitely generated $R/(x)$-modules such that each $P^i$ is free over $R/(x)$.
Decompose it into short exact sequences $\{0\to M^i\to P^i\to M^{i+1}\to0\}_{i\ge0}$ of $R/(x)$-modules, where $M^0=M$.
Taking the syzygy over $R$, we get short exact sequences $\{0\to \syz_RM^i\to Q^i\to \syz_RM^{i+1}\to0\}_{i\ge0}$ of $R$-modules, where $Q^i$ is free over $R$.
Splicing them gives rise to an exact sequence $0\to\syz_RM\to Q^0\to Q^1\to\cdots$ of finitely generated $R$-modules.
Hence $\syz_RM$ is an $\infty$-syzygy over $R$.
If it is totally reflexive over $R$, then $\gdim_RM\le1$ and $\gdim_{R/(x)}M=\gdim_RM-1\le0$ by \cite[(1.2.9) and (1.5.3)]{C}, so that $M$ is totally reflexive over $R/(x)$.
\end{proof}

Next we study a property that unifies both of the properties $\P$ and $\Q$ as in \eqref{6}.
Let $\E$ be a property of local rings.
We consider the following condition.
\begin{equation}\label{12}
\begin{array}{l}
\text{A local ring $R$ satisfies $\E$ if and only if so does the formal power series ring $R[\![X]\!]$,}\\
\phantom{\text{A local ring $R$ satisfies $\E$ }}\text{if and only if so does the completion $\widehat R$.} 
\end{array}
\end{equation}
Note that if $\E$ satisfies the equivalences \eqref{12}, then the implications \eqref{6} hold for $\P=\Q=\E$.
We also consider the following condition.
\begin{equation}\label{13}
\begin{array}{l}
\text{Let $(R,\m)$ be a local ring. Let $x\in\m\setminus\m^2$ be an $R$-regular element.}\\
\phantom{\text{Let $(R,\m)$ be a local ring. }}\text{Then $R$ satisfies $\E$ if and only if so does $R/(x)$.}
\end{array}
\end{equation}
These two conditions have the relationship as stated in the lemma below. 

\begin{lem}\label{15}
If \eqref{13} holds, then \eqref{12} holds as well.
\end{lem}

\begin{proof}
We divide the proof into two steps; combining (1) and (2) below completes the proof of the lemma.

(1) We show that a local ring $R$ satisfies $\E$ if and only if the formal power series ring $R[\![X]\!]$ satisfies $\E$.
Let $\M$ be the maximal ideal of $R[\![X]\!]$.
Then $X$ belongs to $\M\setminus\M^2$ and is $R[\![X]\!]$-regular.
Considering \eqref{13} for the local ring $R[\![X]\!]$, we see that $R[\![X]\!]$ satisfies $\E$ if and only if so does $R[\![X]\!]/(X)\cong R$.

(2) We show that a local ring $R$ satisfies $\E$ if and only if the completion $\widehat R$ satisfies $\E$.
By \cite[Theorem 8.12]{M} we have $\widehat R\cong R[\![X_1,\dots,X_n]\!]/(X_1-a_1,\dots,X_n-a_n)$, where $R[\![X_1,\dots,X_n]\!]$ is a formal power series ring and $a_1,\dots,a_n$ is a system of generators of $\m$.
It is seen from (1) that $R$ satisfies $\E$ if and only if so does $R[\![X_1,\dots,X_n]\!]$.
For each $1\le i\le n$, set $R_i=R[\![X_1,\dots,X_n]\!]/(X_1-a_1,\dots,X_{i-1}-a_{i-1})$ and let $\m_i$ be the maximal ideal of $R_i$.
Then the image of $X_i-a_i$ in $R_i$ belongs to $\m_i\setminus\m_i^2$ and $R_i$-regular.
Considering \eqref{13} for the local ring $R_i$, we observe that $R_i$ satisfies $\E$ if and only if so does $R_i/(X_i-a_i)R_i\cong R_{i+1}$.
Hence $R[\![X_1,\dots,X_n]\!]=R_1$ satisfies $\E$ if and only if so does $R_{n+1}=\widehat R$.
\end{proof}

The following result says that each of the properies that have been introduced so far is preserved under the formal power series extension and the completion.

\begin{thm}\label{14}
The statement \eqref{12} holds for $\E\in\{\ab,\ci,\cm,\ee,\gap,\gor,\tac,\te,\TR\}$.
\end{thm}

\begin{proof}
Proposition \ref{7}(2) implies that if the formal power series ring $R[\![X]\!]$ over a local ring $R$ satisfies $\ee$, then so does $R$, while the converse holds by Proposition \ref{7}(3) and the isomorphism $R[\![X]\!]/(X)\cong R$.
Combining this with Proposition \ref{7}(4) shows that \eqref{12} holds for $\ee$.
An analogous argument using Proposition \ref{7'} instead of Proposition \ref{7} shows that \eqref{12} holds for $\te$.

We shall show that \eqref{13} actually holds for the remaining seven properties; then so does \eqref{12} by Lemma \ref{15}.
It follows from (1) and (3) of \cite[Proposition 3.3]{HJ} that \eqref{13} holds for $\ab$ and $\gap$, respectively.
Fundamental facts say that \eqref{13} holds for $\ci$, $\cm$ and $\gor$.
It is stated in Corollary \ref{16} and Proposition \ref{8}(2) respectively that \eqref{13} holds for $\tac$ and $\TR$.
\end{proof}

\begin{rem}
We do not know whether \eqref{13} holds for $\ee$ or $\te$.
More precisely, we have no idea how to prove that if a local ring $(R,\m)$ satisfies $\ee$/$\te$, then so does $R/(x)$ for each $R$-regular element $x\in\m\setminus\m^2$.
Suppose that $R$ satisfies $\ee$ and let $M,N$ be finitely generated $R/(x)$-modules such that $\Ext_{R/(x)}^{\gg0}(M,N)=0$.
Then, as in \cite[(1.3)]{HJ}, there is an exact sequence $\cdots\to\Ext_{R/(x)}^{i}(M,N)\to\Ext_R^{i}(M,N)\to\Ext_{R/(x)}^{i-1}(M,N)\to\cdots$, which implies $\Ext_R^{\gg0}(M,N)=0$.
Since the ring $R$ satisfies $\ee$, we get $\Ext_R^{\gg0}(N,M)=0$.
As $R$ is Gorenstein by Proposition \ref{7}(1), we see that $\Ext_R^{\gg0}(N,\syz_RM)=0$, and $\Ext_{R/(x)}^{\gg0}(N,\syz_RM/x\syz_RM)=0$ by \cite[\S18, Lemma 2]{M}.
Let $\cdots\to F_1\to F_0\to M\to0$ be a minimal free resolution of $M$ as an $R$-module.
This gives an exact sequence $0\to\syz_RM\to F_0\to M\to0$.
The chain map given by the multiplication by $x$ yields an exact sequence $\sigma:0\to M\to\syz_RM/x\syz_RM\to\syz_{R/(x)}M\to0$ of $R/(x)$-modules.
If $\sigma$ splits, then we will get that $\Ext_{R/(x)}^{\gg0}(N,M)=0$ and conclude that $R$ satisfies $\ee$.
However, $\sigma$ does not split in general.
\end{rem}

\section{On a question of Araya and the depth formula}

In a private communication of Justin Lyle with Tokuji Araya, the following question is posed.

\begin{ques}[Araya]\label{18}
Let $R$ be a local ring of positive depth.
Let $M$ be a finitely generated $R$-module such that $\Ext_R^{>0}(M,R)=0$.
Then must $M$ have positive depth?
\end{ques}

We relate this question with the property $\TR$.
For this we prepare two lemmas.
The first one is elementary.

\begin{lem}\label{21}
Let $R$ be a local ring with maximal ideal $\m$ such that $\depth R>0$.
Let $M$ be a finitely generated $R$-module such that $\depth M>0$ and $M\in\mod_0R$.
Then $M$ is torsionless.
\end{lem}

\begin{proof}
Fix a prime ideal $\p$ of $R$.
First, assume that $\p\in\ass R$.
Then $\p\ne\m$, and hence $M_\p$ is $R_\p$-free.
In particular, the $R_\p$-module $M_\p$ is torsionless.
Next, assume that $\depth R_\p>0$.
Suppose that $\depth M_\p=0$.
Then $\p\ne\m$, and hence $M_\p$ is $R_\p$-free.
We have $0=\depth M_\p=\depth R_\p>0$, which is a contradiction.
Therefore, we must have $\depth M_\p>0$.
It follows from \cite[Proposition 1.4.1]{BH} that $M$ is torsionless.
\end{proof}

The second lemma is shown by a similar argument as in the proof of \cite[Theorem 4.3]{res}.
To prove the lemma, we need to recall some notation and terminology.
We denote by $\mod R$ the category of finitely generated $R$-modules, and by $\mod_0R$ the full subcategory of $\mod R$ consiting of modules which are locally free on the punctured spectrum of $R$.
A {\em resolving subcategory} of $\mod R$ is by definition a full subcategory of $\mod R$ containing $R$ and closed under direct summands, extensions and syzygies.
For a finitely generated $R$-module $M$, we denote by $\res M$ the {\em resolving closure} of $M$, that is, the smallest resolving subcategory of $\mod R$ containing $M$, and by $\nf(M)$ the {\em nonfree locus} of $M$, that is, the set of prime ideals $\p$ of $R$ such that the $R_\p$-module $M_\p$ is nonfree.
The subset $\nf(M)$ of $\spec R$ is Zariski-closed; see \cite[Corollary 2.11]{res}.

\begin{lem}\label{20}
Let $(R,\m)$ be a local ring with $\depth R>0$.
Let $M$ be an $R$-module such that $\depth M>0$ and $\gdim M=\infty$.
Then there exists an $R$-module $N\in\res M\cap\mod_0R$ such that $\depth N>0$ and $\gdim N=\infty$.
\end{lem}

\begin{proof}
If $M\in\mod_0R$, letting $N=M$ completes the proof.
Assume $M\notin\mod_0R$.
Then there is a prime ideal $\p$ of $R$ with $\m\ne\p\in\nf(M)$.
We find an $R$-regular element $x\in\m\setminus\p$.
Consider the pushout diagram:
$$
\xymatrix{
0\ar[r]& \syz M\ar[r]\ar[d]^x& R^{\oplus n}\ar[r]\ar[d]& M\ar[r]\ar@{=}[d]& 0\\
0\ar[r]& \syz M\ar[r]& N\ar[r]& M\ar[r]& 0
}
$$
The bottom row shows $\depth N>0$ and $N\in\res M$.
We observe $\nf(N)\subseteq\nf(M)$ and $N_\p\cong R_\p^{\oplus n}$.
Hence $\nf(N)\subsetneq\nf(M)$.
Applying the snake lemma to the diagram, we get an exact sequence $0\to R^{\oplus n}\to N\to\syz M/x\syz M\to0$.
The bottom row again shows that if $\gdim N<\infty$, then $\gdim M<\infty$, a contradiction.
Hence $\gdim N=\infty$.
If $N\in\mod_0R$, then we are done.
If $N\notin\mod_0R$, then applying the above argument to $N$ gives us an $R$-module $L\in\res N\subseteq\res M$ such that $\nf(L)\subsetneq\nf(N)$, $\depth L>0$ and $\gdim L=\infty$.
It is impossible to repeat this argument infinitely many times, because otherwise we get a chain $\cdots\subsetneq\nf(L)\subsetneq\nf(N)\subsetneq\nf(M)$ of Zariski-closed sets, which contradicts the fact that $\spec R$ is a noetherian space.
\end{proof}

We denote by $\add R$ the full subcategory of $\mod R$ consisting of free modules, and by $(-)^\ast$ the $R$-dual functor $\Hom_R(-,R)$.
The {\em first cosyzygy} $\syz^{-1}M$ of a finitely generated $R$-module $M$ is defined as the cokernel of a {\em left $\add R$-approximation} (or {\em $\add R$-preenvelope}) $f:M\to F$, that is, $f$ is a morphism in $\mod R$ with $F$ free such that $f^\ast:F^\ast\to M^\ast$ is surjective.
For an integer $n\ge2$ the {\em $n$th cosyzygy} $\syz^{-n}M$ is defined inductively by $\syz^{-n}M=\syz^{-1}(\syz^{-(n-1)}M)$.
It is known that $\syz^{-1}M$ always exists and is isomorphic to $\tr\syz\tr M$ up to free summands, so that $\Ext_R^1(\syz^{-1}M,R)=0$.
For the details of cosyzygies, we refer the reader to \cite{AB,syz}.

In the theorem below, we obtain characterizations of local rings of positive depth that satisfy $\TR$.

\begin{thm}\label{22}
Let $R$ be a local ring with $\depth R>0$.
Then the following are equivalent.
\begin{enumerate}[\rm(1)]
\item
Every finitely generated $R$-module $M$ with $\Ext_R^{>0}(M,R)=0$ is totally reflexive, that is, $R$ satisfies $\TR$.
\item
Every finitely generated $R$-module $M$ with $\Ext_R^{>0}(M,R)=0$ satisfies $\depth M>0$.
\end{enumerate}
In particular, Question \ref{18} has an affirmative answer for $R$ if and only if $R$ satisfies $\TR$.
\end{thm}

\begin{proof}
It is clear that (1) implies (2).
Assume that (1) does not hold but (2) does.
Then there is an $R$-module $M$ which is not totally reflexive but satisfies $\Ext_R^{>0}(M,R)=0$.
It is seen from \cite[(1.2.7)]{C} that $\gdim M=\infty$.
By (2) we have $\depth M>0$.
Also, $M$ belongs to the full subcategory $\X$ of $\mod R$ consisting of modules $X$ with $\Ext_R^{>0}(X,R)=0$.
As $\X$ is resolving, by Lemma \ref{20} we find an $R$-module $N\in\X\cap\mod_0R$, $\depth N>0$ and $\gdim N=\infty$.
Lemma \ref{21} implies that $N$ is torsionless.
Therefore, there is an exact sequence $0\to N\to F^0\to\syz^{-1}N\to0$ with $F^0$ free.
This exact sequence implies $\gdim\syz^{-1}N=\infty$ and $\Ext^{>0}(\syz^{-1}N,R)=0$, while we see that $\syz^{-1}N\in\mod_0R$.
By (2) again we have $\depth\syz^{-1}N>0$.
Applying Lemma \ref{21} again, we see that $\syz^{-1}N$ is torsionless, and get an exact sequence $0\to\syz^{-1}N\to F^1\to\syz^{-2}N\to0$ with $F^1$ free.
Iterating this procedure yields an exact sequence
$$
0\to N\to F^0\xrightarrow{\partial^1}F^1\xrightarrow{\partial^2}F^2\xrightarrow{\partial^3}\cdots
$$
such that for each $i>0$ we have that $F^i$ is free, the image of $\partial^i$ is $\syz^{-i}N$, and $\Ext^{>0}(\syz^{-i}N,R)=0$.
Applying the functor $(-)^\ast$ gives rise to an exact sequence $\cdots\to(F^2)^\ast\to(F^1)^\ast\to(F^0)^\ast\to N^\ast\to0$, and applying $(-)^\ast$ again restores the original exact sequence.
This shows that $N$ is totally reflexive.
However, this contradicts the fact that $\gdim N=\infty$.
We now conclude that (2) implies (1).
\end{proof}

We can show the following proposition on the relationship between $\dep$ and $\TR$.

\begin{prop}\label{17}
Let $R$ be a Cohen--Macaulay local ring of positive dimension.
If $R$ satisfies $\dep$, then it satisfies $\TR$.
\end{prop}

\begin{proof}
Suppose that $R$ does not satisfy $\TR$.
Then there exists an $R$-module $M$ such that $\Ext_R^{>0}(M,R)=0$ and $\depth M=0$ by Theorem \ref{22}.
Since $\widehat R$ is complete, it admits a canonical module $\omega$.
Take a system of parameters $\xx=x_1,\dots,x_d$ of $R$, where $d=\dim R>0$.
Note that $\xx$ is an $R$-sequence, an $\widehat R$-sequence and an $\omega$-sequence.
We have $\Ext_{\widehat R}^{>0}(\widehat M,\widehat R)=0$.
It follows from \cite[Lemma 3.4]{LM} that $\Tor_{>0}^{\widehat R}(\widehat M,\omega)=0$ and $\widehat M\otimes_{\widehat R}\omega$ is maximal Cohen--Macaulay.
Hence $\xx$ is an $\widehat M\otimes_{\widehat R}\omega$-sequence.
There are isomorphisms
\begin{align*}
M\lten_R\omega/\xx\omega
&\cong M\lten_R\omega\lten_{\widehat R}\widehat R/\xx\widehat R
\cong M\lten_R\widehat R\lten_{\widehat R}\omega\lten_{\widehat R}\widehat R/\xx\widehat R
\cong(\widehat M\lten_{\widehat R}\omega)\lten_{\widehat R}\widehat R/\xx\widehat R\\
&\cong(\widehat M\otimes_{\widehat R}\omega)\lten_{\widehat R}\widehat R/\xx\widehat R
\cong(\widehat M\otimes_{\widehat R}\omega)\otimes_{\widehat R}\widehat R/\xx\widehat R
\cong M\otimes_R\omega/\xx\omega.
\end{align*}
We obtain $\Tor_{>0}^R(M,\omega/\xx\omega)=0$.
Since $\omega/\xx\omega$ has finite length as an $\widehat R$-module, it has finite length as an $R$-module.
In particular, $\omega/\xx\omega$ is a finitely generated $R$-module.
As $R$ satisfies $\dep$, we have
$$
0\le\depth_R(M\otimes_R\omega/\xx\omega)=\depth_RM+\depth_R \omega/\xx\omega-\depth R=0+0-d=-d<0.
$$
This contradiction completes the proof of the proposition.
\end{proof}

In the theorem below we investigate the relationship between $\dep$ and the vanishing of Ext modules.

\begin{thm}\label{26}
Let $(R,\m,k)$ be a Cohen--Macaulay local ring of dimension $d>0$ with a canonical module $\omega$.
If $R$ satisfies $\dep$, then for all two finitely generated $R$-modules $M$ and $N$ such that $\Ext_R^{\gg0}(M,N)=0$ one has $\Ext_R^{>d}(M,N)=0$.
The converse holds true if $R$ is Gorenstein.
\end{thm}

\begin{proof}
A Gorenstein local ring satisfying the condition in the assertion about modules $M,N$ is an AB ring.
It is shown in \cite[Corollary 5.3(b)]{CJ} that every AB ring satisfies $\dep$.

Assume that $R$ satisfies $\dep$.
Let $n\ge0$ be an integer, and let $M$ and $N$ be finitely generated $R$-modules such that $\Ext_R^{>n}(M,N)=0$ and $\Ext_R^{n}(M,N)\ne 0$.

We claim that if $M$ and $N$ are maximal Cohen--Macaulay $R$-modules, then $n=0$.
In fact, assume $n>0$.
Let $\xx=x_1,\dots,x_d$ be a system of parameters of $R$.
This is a regular sequence on $K:=\syz^{n-1}M$.
Since $\Ext_R^{>1}(K,N)=0$ and $\Ext_R^1(K,N)\ne 0$, an analogous argument as in the proof of Lemma \ref{1}(2) shows that $\Ext_R^{>d+1}(K/\xx K,N)=0$ and $\Ext_R^{d+1}(K/\xx K,N)\ne 0$.
The $d$th syzygy $L=\syz_R^d(K/\xx K)$ is a maximal Cohen--Macaulay $R$-module such that $\Ext_R^{>1}(L,N)=0$ and $\Ext_R^1(L,N)\ne 0$.
As $L$ is locally free on the punctured spectrum of $R$ and $d>0$, by \cite[Lemma 3.5(2)]{KOT} we get $\Tor_i^R(L,N^\dag)\cong\Ext_R^{d+i}(L,N)^\vee=0$ for all $i>0$, where $(-)^\dag=\Hom_R(-,\omega)$ is the canonical dual and $(-)^\vee=\Hom_R(-,\EE_R(k))$ is the Matlis dual.
Since $R$ satisfies $\dep$, we have $\depth(L\otimes_RN^\dag)=\depth L+\depth N^\dag-\depth R=d$, and hence $L\otimes_RN^\dag$ is maximal Cohen--Macaulay.
It follows from \cite[Proposition 2.5]{KOT} that $\Ext_R^{i}(L,N)=0$ for all integers $1\le i\le d$.
This contradicts the fact that $\Ext_R^1(L,N)\ne 0$; recall that $d>0$.
Thus we must have $n=0$.
The claim follows.

As $R$ is a Cohen--Macaulay local ring with a canonical module, there exists a {\em maximal Cohen--Macaulay approximation} of $N$, that is to say, an exact sequence $0\to Y\to X\to N\to 0$ of finitely generated $R$-modules such that $Y$ has finite injective dimension and $X$ is maximal Cohen--Macaulay; see \cite[Theorem 11.17]{LW}.
Since $H:=\syz^dM$ is maximal Cohen--Macaulay, we have $\Ext_R^{>0}(H,Y)=0$.
We see that $\Ext_R^{i}(H,X)\cong\Ext_R^{i}(H,N)$ for all $i>0$, and that $\Ext_R^i(H,N)\cong\Ext_R^{i+d}(M,N)=0$ for all $i\gg0$.
The above claim implies $\Ext_R^{>0}(H,X)=0$, whence $\Ext_R^{>0}(H,N)=0$ and $\Ext_R^{>d}(M,N)=0$.
This completes the proof of the theorem.
\end{proof}

\section{Proofs of the main results}

We give proofs of the five results stated in the Introduction.
The first theorem is shown by a short proof and independent of the results given in the previous sections.
A theorem of Heitmann \cite{He} plays a key role.

\begin{proof}[Proof of Theorem \ref{0}]
(1) Consider the formal power series ring $B=A[\![X_1,\dots,X_n]\!]$, where $n:=\rho-\alpha\ge0$.
Then $B$ is complete (see \cite[Exercise 8.6]{M}), $B/\m_B\cong A/\m_A=k$, $\ch B=\ch A=\ch k$, and $\depth B=\rho\ge2$.
Since $A$ contains a field, any integer $m$ such that $m\cdot1\ne0$ in $R$ is a non-zerodivisor of $R$.
By virtue of \cite[Main Theorem]{He}, there exists a local unique factorization domain $R$ having an isolated singularity such that $\widehat R\cong B$.
We have $\depth R=\rho$, $\ch R=\ch B$ and $R/\m_R\cong B/\m_B$.
As $A$ does not satisfy $\P$, neither does $B$, and neither does $R$.
Since $A$ satisfies $\Q$, so does $B$, and so does $R$.

(2) Apply the proof of (1) to $n:=\sigma-\alpha\ge0$; note that $\depth B=\sigma\ge1$.
Then we find a local domain $S$ having an isolated singularity such that $\widehat S\cong B$, and observe that $S$ does not satisfy $\P$ but satisfies $\Q$, $\depth S=\sigma$ and $\ch S=\ch A=\ch k$.
\end{proof}

\begin{proof}[Proof of Theorem \ref{27}]
The assertion immediately follows from Theorem \ref{26} and \cite[Proposition 3.2]{HJ}.
\end{proof}

\begin{proof}[Proof of Corollary \ref{4}]
By virtue of \cite[Corollary 4.2]{JS5}, there exists an artinian Gorenstein equicharacteristic local ring $A$ which does not satisfy $\ee$.
Note then that $A$ is complete.
Set $\P=\ee$ and $\Q=\gor$.
Theorem \ref{14} says that $\P$ and $\Q$ satisfy the implications \eqref{6}.
The assertion follows from Theorem \ref{0}.
The last assertions of (1) and (2) are consequences of the first ones and Theorem \ref{27}.
\end{proof}

\begin{proof}[Proof of Theorem \ref{29}]
The assertion is an immediate consequence of Propositions \ref{9} and \ref{17}.
\end{proof}

\begin{proof}[Proof of Corollary \ref{5}]
By \cite[Theorem 1.7]{JS6}, there is an artinian equicharacteristic local ring $A$ which does not satisfy $\TR$.
Then $A$ is complete and Cohen--Macaulay.
Let $\P=\TR$ and $\Q=\cm$.
Theorem \ref{14} says $\P,\Q$ satisfy \eqref{6}.
The first assertions of (1) and (2) follow by Theorem \ref{0}.
The last assertions of (1) and (2) are consequences of the first ones and Theorem \ref{29}.
\end{proof}

Combining Theorem \ref{22} with Corollary \ref{5}, we obtain the following result.

\begin{cor}
Let $k$ be a field which is not algebraic over a finite field.
\begin{enumerate}[\rm(1)]
\item
For each $d\ge2$, there exists a $d$-dimensional Cohen--Macaulay equicharacteristic local unique factorization domain $(R,\m_R,k)$ with an isolated singularity which admits a finitely generated module $M$ such that $\Ext_R^{>0}(M,R)=0$ and $\depth M=0$.
\item
There exists a $1$-dimensional Cohen--Macaulay equicharacteristic local domain $(S,\m_S,k)$ which admits a finitely generated module $N$ such that $\Ext_S^{>0}(N,S)=0$ and $\depth N=0$.
\end{enumerate}
In particular, both {\rm(1)} and {\rm(2)} give negative answers to Question \ref{18}.
\end{cor}

In a similar way, it is actually possible to get more results in the same context.
For instance, let $A=k[X,Y,Z]/(X^2-Y^2,X^2-Z^2,XY,XZ,YZ)$ be a residue ring of a polynomial ring over a field $k$.
Then $A$ is an artinian Gorenstein non-complete intersection equicharacteristic local ring with embedding dimension $3$ and multiplicity $5$.
It is seen from \cite[Theorem 3.4(2) and Lemma 3.7]{HJ} that $A$ satisfies $\gap$ and $\te$.
Applying Theorems \ref{0} and \ref{14} to this ring $A$ and the properties $\P=\ci$ and $\Q=\gap\land\te$ yields:

\begin{cor}
Let $k$ be a field.
Then the following two statements hold true.
\begin{enumerate}[\rm(1)]
\item
For any $d\ge2$, there exists a $d$-dimensional equicharacteristic local unique factorization domain $(R,\m_R,k)$ with an isolated singularity which is not a complete intersection but satisfies both $\gap$ and $\te$.
\item
There exists a $1$-dimensional equicharacteristic local domain $(S,\m_S,k)$ which is not a complete intersection but satisfies both $\gap$ and $\te$.
\end{enumerate}
\end{cor}

No counterexample has been found so far to each of the implications $\gap\Rightarrow\ab\Rightarrow\ee\Leftarrow\te$ and $\tac\Rightarrow\TR$.
Once a counterexample of an artinian equicharacteristic local ring is found, one can lift it to a counterexample of a (unique factorization) domain with an isolated singularity by Theorems \ref{0} and \ref{14}.

Finally, we give some comments on \cite{Y}.

\begin{rem}
Corollary \ref{5} says that the assertions of \cite[Theorem 1.1 and Corollary 1.3]{Y} are both incorrect.
In their proofs, \cite[Theorem 8.5]{Y} plays an essential role, and the authors wonder if the proof of \cite[Theorem 8.5]{Y} contains gaps.
To be more precise, the following sentence is given in \cite[page 132, lines 8--9]{Y}.
\begin{quote}
Similarly, $\varphi_n^F : \widetilde{F}\to X_0$
is represented by $(0\ \varphi_{n-1}^F):F_{n-1} [n-1]\oplus \widetilde{F^\prime}\to X_0$.
\end{quote}
We are not sure why this sentence is true.
If it were true, by induction $\varphi_n^F$ would be represented by $(0\ 0\ \cdots\ 0\ p_0^F): F_{n-1}[n-1]\oplus F_{n-2}[n-2]\oplus\cdots\oplus F_1[1]\oplus F_0 \to
X_0$, but we feel that this claim is too strong.
In fact, for example, in \cite[Lemma 9.4]{Y}, to deduce that $\varphi_n^F$ has such a representation, a very strong assumption is imposed on those complexes for which contraction is taken.
Other than this, a lot of identifications are done in the proof of \cite[Theorem 8.5]{Y}, about whose correctness we are not sure.
The classical proof that the octahedral axiom holds in the homotopy category would give that $\varphi_n^F:\widetilde{F}\to X_0$ is represented by
$$
(p_{n-1}^F[n-1]\ \ \varphi_{n-1}^F):F_{n-1}[n-1]\oplus\widetilde{F'}\to X_0
$$
instead of $(0\ \varphi_{n-1}^F)$, noting one of the identifications here is to identify $X_0$ with $\cone(\psi_{n-1}^F)$ which as an underlying graded $R$-module is $X_{n-1}[n-1]\oplus\widetilde{F'}$.
To be more precise, inductively
there is a diagram:
$$
\xymatrix{
X_{n-1}[n-2]\ar[r]^-{\psi_{n-1}^F}\ar@{=}[d]& \widetilde{F'}\ar[r]^-{\varphi_{n-1}^F}\ar@{=}[d]& X_0\ar[r]^{\widetilde{\omega_{n-1}^F}}\ar@<-2mm>[d]_-{\delta_{n-1}^F}& X_{n-1}[n-1]\ar@{=}[d]\\
X_{n-1}[n-2]\ar[r]^-{\psi_{n-1}^F}& \widetilde{F'}\ar[r]^-{\lambda_{n-1}^F}& \cone(\psi_{n-1}^F)\ar[r]^-{\mu_{n-1}^F}\ar@<-2mm>[u]_-{\epsilon_{n-1}^F}& X_{n-1}[n-1]
}
$$
where $\delta_{n-1}^F$ and $\epsilon_{n-1}^F$ are homotopy inverses and where the bottom row is the natural strict
triangle, which identifies the top row as a triangle in $\K(R)$.
Then the map $\varphi_n^F$ is given by $(\varepsilon_{n-1}^F\circ\binom{p_{n-1}^F[n-1]}{0}\ \ \varphi_{n-1}^F)$ from the standard proof that the octahedral axiom holds, but importantly, with this description of $\varphi_n^F$, the diagram in the proof of \cite[Theorem 8.5]{Y} does not obviously commute, and should not commute in general.
At least many of the remaining identifications made in the proof of \cite[Theorem 8.5]{Y} could be similarly avoided by keeping track of the relevant isomorphisms between the given triangles and their strict counterparts.
So the key point is that the correct inductive description of the map $\varphi_n^F$ is one where the relevant diagram does not commute in general.

On the other hand, the ring $R$ produced by Corollary \ref{5} is not excellent.
So, even if the proof of \cite[Theorem 8.5]{Y} contains gaps, the assertion itself may be true in the case where the base ring is excellent.
However, the theory developed in \cite{Y} does not seem to be related to the excellence of the base ring, so even if the assertion of \cite[Theorem 8.5]{Y} is true for excellent rings, we would need another approach to show it.
\end{rem}


\end{document}